\documentclass[11pt]{elsarticle}

\usepackage{footnote}
\usepackage{enumerate}

\usepackage{ulem}
\usepackage{amsfonts,amsthm,amssymb,amsmath,url}
\usepackage{graphicx,color}
\usepackage{mathrsfs}
\usepackage[colorlinks]{hyperref}

\newtheorem{theorem}{Theorem}

\newtheorem{assumption}{Assumption}

\newtheorem{definition}{Definition}
\theoremstyle{definition}
\newtheorem{example}{Example}

\newcommand{\E}{\mathbb{E}}

\newcommand{\R}{\mathbb{R}}
\newcommand{\Z}{\mathbb{Z}}
\newcommand{\ignore}[1]{}

\title{Identifying Effective Scenarios for Sample Average
Approximation}
\date{\today}
\author{
Lijian Chen\\
 Department of MIS, Operations Management, and Decision Sciences\\
 University of Dayton\\
 {\tt lchen1@udayton.edu}\\
}

\begin{document}

\maketitle
{\bf ABSTRACT}\ \\

We introduce a method to improve the tractability of the
well-known Sample Average Approximation (SAA) without compromising
important theoretical properties, such as convergence in probability and
the consistency of an independent and identically distributed (iid)
sample. We consider each scenario as a polyhedron of the mix of
first-stage and second-stage decision variables. According to 
John's theorem, the L\"owner-John ellipsoid of each polyhedron will be
unique which means that different scenarios will have correspondingly
different L\"owner-John ellipsoids. By optimizing the objective
function regarding both feasible regions of the polyhedron and its
unique L\"owner-John ellipsoid, respectively, we obtain a pair of
optimal values, which would be a coordinate on a two-dimensional plane. The
scenarios, whose coordinates are close enough on the plane, will be treated as one scenario; thus our
method reduces the sample size of an iid sample considerably. Instead
of using a large iid sample directly, we would use the cluster of low-cost
computers to calculate the coordinates of a massive number of scenarios
and build a representative and significantly smaller sample to
feed the solver. We show that our method will obtain the optimal
solution of a very large sample without compromising the 
solution quality. Furthermore, our method would be implementable as a
distributed computational infrastructure with many but low-cost
computers.


{\bf HISTORY:} This paper was 
first submitted on March 21, 2019.


%


\section{Introduction}
\label{sec:introduction}

A standard formulation of the two-stage stochastic program is
\begin{equation}
  \label{eq:1}
  \min_x\{c^Tx+\E[Q(x,\xi)],x\in X\subset \R^n\}
\end{equation} where
\begin{equation}
  \label{eq:2}
  Q(x,\xi):=\inf_y\{q^Ty:Wy\geq h-Tx, y\in Y\subset \R^m\}\nonumber
\end{equation}
and $\xi$ represents a random vector whose distribution has finite support. $W$ is a deterministic matrix, 
while $h$ and $T$ would be functions with respect to $\xi$. If $X\subset \Z^n,$ and
$Y\subset \Z^m$, the problem becomes a stochastic integer program.

A common practice to numerically solve stochastic (integer) program is
by sampling. In \cite{kleywegt2002sample}, the authors generate
iid realizations, $\xi^1,\ldots,\xi^N$, of the random vector $\xi$ and
approximate the objective function $\E[Q(x,\xi)]$ with its average
$\dfrac{1}{N}\sum_{i=1}^NQ(x,\xi^i)$. The stochastic (integer) program
becomes a deterministic model as follows:
\begin{equation}
  \label{eq:3}
  \min_x\{c^Tx+\dfrac{1}{N}\sum_{i=1}^NQ(x,\xi^i),x\in X\}
\end{equation}
Let $\hat\nu_N$ and $\hat x_N$ denote the optimal value and the
optimal solution of \eqref{eq:3}, respectively while $\nu^*$ and $x^*$
represent the optimal value and the optimal solution of
\eqref{eq:1}. The authors showed that both $\hat \nu_N$ and $\hat x_N$
will converge to their counterparts as the sample size $N$ becomes
large enough, regardless of the distribution of $\xi$. These 
results are referred to as convergence in probability and
consistency. This method is referred to as the Sample Average Approximation
(SAA) and in \cite{homem2008rates}, we have the convergence results
regarding a few sampling methods. We find the recent advance of SAA in \cite{homem2014monte} and
\cite{rahimian2018identifying} with references therein.  

The computational cost is entirely determined by the sample size $N$. The
theoretical argument in determining the sample size $N$ is based on
the Large Deviation theory (see \cite{shapiro2009lectures}). According
to this theory, there is a high probability that the values of the
sample average approximation and the true function are close to each
other at a sufficiently dense set of points. The result is impractical
because the required sample size has been one of the primary
impediments to SAA's implementation in practice. Moreover, in \cite{dyer2006computational}, the stochastic 
program is recognized as \#P-hard, which indicates that it is computationally
intractable. The term \#P-hard, rather than NP-hard, is used to
describe the fact that the computer hardware will be overwhelmed by the
number of scenarios required to complete the numerical method. For the 
stochastic integer program, it becomes worse because a large sample
size will lead to an unrealistically large deterministic integer
program.

Efforts to reduce the sample size may introduce the SAA method to many
scenario-rich industries, such as financial planning and
logistics. There are a few major ideas: the scenario reduction method, 
such as \cite{dupavcova2003scenario}, the variance reduction method in
\cite{xiao2014proximal}, and the Quasi-Monte-Carlo method in
\cite{leovey2015quasi}. There are many articles prior to the papers
mentioned above. In \cite{kim2015guide}, the authors present a
comprehensive review on the SAA, such that a moderately large sample is more
likely to be satisfactory for some problems. Thus, when a practitioner encounters a
stochastic (integer) program, the choice would be either to hope for a
moderately large sample that will deliver a satisfactory numerical solution, or
to adopt the reduction methods or the Quasi-Monte-Carlo method with 
non-iid samples.

The SAA method with a moderately large sample \emph{could} be either 
satisfactory or otherwise because the sample size requirement imposed
by the Large Deviation theory will never be met in practice. The
obtained solution may still converge to the true optimal solution, while the variation of the
obtained optimal value could be concerning. A viable solution to 
reducing the variation of the optimal solution is to increase the sample
size of the iid sample, which will lead to intractability. Reduction
methods, such as scenario reduction, variance  
reduction, and the Quasi-Monte-Carlo, aim to sample the random variable in an
artificially defined manner. These methods will effectively reduce the
sample size and assure tractability. However, since the scenarios
of these methods are no longer independently generated, the obtained optimal solution
may exhibit, at least theoretically, bias with respect to the true
optimal solution and the distribution of the sample average will no
longer be normal. Therefore, the situation is
that if we use the SAA method, we may have a computationally challenging optimization
problem, while we may incur an inconsistent optimal solution when adopting
the reduction methods.

In this paper, we study a method to preserve the convergence in
probability, the consistency regarding the optimal solution and to
control computational tractability. The idea is to first 
generate a {\it large enough} Independent and Identically Distributed (iid)
sample. Instead of using this large iid sample directly, we will
attach values to each sampled scenario as a measure of similarity, and
we will cluster ``similar'' scenarios as a representative scenario
with an adjusted probability to formulate a new, but reduced sample. We
show that the newly formed sample will yield a consistent optimal
solution with a bounded difference to the counterpart of the original
large iid sample. 

The essential result is to define the measure of similarity. We
realize that each scenario is described as a polyhedron if not
empty. To measure the similarity of polyhedrons of the same dimension,
we need to compare them as distinct geometric objects, which is
difficult due to the exploding number of vertices. We have to
transform polyhedrons into comparable geometric objects, and the most
well-known object is the L\"owner-John ellipsoid, which is the maximum volume inscribed ellipsoid that
is contained in the polyhedron. The L\"owner-John ellipsoid is unique to each
convex body (\cite{john2014extremum}). In other words, the mapping
from the polyhedron to its L\"owner-John ellipsoid is a one-to-one
correspondence. For the same objective function regarding different 
feasible regions such as the polyhedron and its L\"owner-John
ellipsoid, the pair optimal values would be considered as a coordinate on a
two-dimensional plane. We define the similar scenarios as the
scenarios whose coordinates stay close and our
method will identify similar scenarios and consolidate them as one
representative scenario to achieve the goal of reducing the sample
size. 

Our method is both theoretically and computationally plausible. The
L\"owner-John ellipsoid is a good choice when approximating
polyhedrons of distinct scenarios because the mapping from the
polyhedron to its L\"owner-John ellipsoid 
is a one-to-one correspondence. We realize that distinct scenarios will have distinct
L\"owner-John ellipsoids, which are strictly convex. When the
objective function is convex and we use the L\"owner-John ellipsoid as
the feasible region, the optimal values associated with distinct
ellipsoids  will be distinct as well. We also obtain the optimal values of the same
objective function subject to the polyhedron as feasible region. Thus,
we can scalarize distinct scenarios into
distinct pairs of coordinates and the seemingly different geometric objectives become
comparable as the coordinates on a two-dimensional plane. 

The efficient computational methods of L\"owner-John
 ellipsoids are available; see \cite{sun2004computation} for more
details. There are multiple software packages, such as SeDuMi and
SDPA, which are ready for implementations. Furthermore, since our method will
evaluate the similarity of each scenario, it can be implemented
on multiple, low-cost computers deployed in parallel to process a
large iid sample. The benefit of our method is to reduce the problem
scale to its fractions to improve the tractability of SAA of large
sample. The remainder of the paper is organized as 
follows. We present the connection between a sampled scenario and an 
ellipsoid in Section \ref{sec:ellipsoid}. In Section
\ref{sec:measuring-similarity}, we show the computation of the measure
of similarity among distinct scenarios, followed by details regarding the
clustering method, along with the result of preserving of convergence
in probability and consistency of the SAA in Section
\ref{sec:clust-simil-scen}. In Section 
\ref{sec:numerical-results}, we present multiple examples to show
the numerical results, and we conclude our research in Section
\ref{sec:conclusion}. 

\section{L\"owner-John ellipsoid associated with each scenario}
\label{sec:ellipsoid}

In this section, we present the approximation of a scenario using the
L\"owner-John ellipsoid. Suppose that we have a two-stage stochastic
(integer) programming. The distribution of $\xi$ would be either
continuous or discrete in terms of finite support. Additionally, there would be finitely
 or infinitely many realizations. The problem is solved by the SAA with a sample size $K$, i.e.,
$\xi^1,\ldots,\xi^k,\ldots,\xi^K$ are $K$ realized scenarios:
\begin{equation}
  \label{eq:8}
  \min_x\{c'x+\dfrac{1}{K}\sum_{k=1}^KQ(x,\xi^k)]|Ax\leq b, x\geq 0\}\nonumber
\end{equation} and
\begin{equation}
  \label{eq:9} 
  Q(x,\xi^k)=\min_y\{q_k'y|Wy\leq h_k-T_kx, y\geq 0\}\nonumber
\end{equation} where the matrices $q_k, T_k, h_k \in \R^{m}$ would be
functions of $\xi$, and $W$ is a fixed $\ell\times m$ matrix. $Q(x,\xi)$ refers to
the recourse function. 

We assume that when the first-stage decisions are discrete, the set of first-stage
decisions is finite and non-empty. Also, the recourse function
$Q(x,\xi)$ is measurable and both $\E[Q(x,\xi)]$ and $\E[Q^2(x,\xi)]$
are finite for every feasible $x$. Otherwise, i.e., if the first-stage
decisions are continuous, the set of first-stage decisions is
non-empty, compact, and polyhedral. Also, the recourse function
$Q(x,\xi^k)$ is finite for $k=1,\ldots,K$. In addition, we assume that for any
first-stage decision, the feasible set of the recourse variable $y$ is
non-empty and finite. 

The SAA method solves the following deterministic optimization
problem.
\begin{alignat}{2}
  \label{eq:5}
  \min_{x,y_1,\ldots,y_K}&\
  c'x+\dfrac{1}{K}\sum_{k=1}^Kq_k'y_k\nonumber\\
  \textrm{subject to}&\ Ax\leq b, x\geq 0, x\in X\subset
  \R^n\\
  &\ Wy_k\leq h_k-T_kx, y_k\geq 0, k=1,\ldots,K.\nonumber
\end{alignat} and the first-stage optimal solution is $x_K^*$ with the optimal
value $\nu^*_K$. We are interested in the case of only one scenario,
$\xi^k, k\in \{1,\ldots,K\}$, which is selected. We have 
\begin{alignat}{2}
\label{eq:10}
  \kappa(k):=\min_{x_k,y_k}&\
  c'x_k+q_k'y_k\nonumber\\
  \textrm{subject to}&\ Ax_k\leq b, x_k\geq 0, x_k\in X\subset
  \R^n\\
  &\ Wy_k\leq h_k-T_kx_k, y_k\geq 0.\nonumber
\end{alignat} where $x_k$ and $y_k$ are the first-stage and
second-stage decisions, respectively. We define that for a non-empty
and bounded sampled
scenario $\xi^k$, 
\begin{definition}
  \label{sec:ellipsoid-1}
  \begin{equation}
    \label{eq:11}
    P_k:=\{[x_k;y_k]|Wy_k\leq h_k-T_kx_k, y_k\geq 0\}
  \end{equation} is the $(n+m)$ dimensional scenario polyhedron of the
  scenario $\xi_k$.   
\end{definition}
We note that if the decisions are integers, i.e., $x_k,y_k\in\Z^{n+m}$, we
need to relax the integer constraints to $x_k,y_k\in\R^{n+m}$ to
assure that $P_k$ a non-empty and bounded polyhedron.

\begin{definition}
  \label{sec:ellips-assoc-with-1}
  An ellipsoid $E$ of $\R^{n+m}$ is an affine image of the unit ball
  $B_{n+m}:=\{u\in\R^{n+m}:||u||\leq 1\}$, that is
  \begin{equation}
    \label{eq:12}
    E:=\{c+Su:u\in\R^{n+m},||u||\leq 1\}, \textrm{or
    }E:=\{x\in\R^{n+m}:||S^{-1}(x-c)||\leq 1\}\nonumber
  \end{equation} where $S\in\R^{(n+m)\times(n+m)}$ is a symmetric,
  non-singular, and positive definite matrix. 
\end{definition}
Also, we need the following assumption.
\begin{assumption}
  \label{sec:ellips-assoc-with}
  $P_k$ satisfies the Slater's condition. 
\end{assumption}
This assumption is rather important because otherwise, a bounded and non-empty
$P_k$ could be a degenerate polyhedron, which has zero volume for
L\"owner-John ellipsoid. More importantly, many two-stage stochastic
programming has equality constraints in the second stage. For
example, in the aircraft allocation problem in
\cite{dantzig2016linear}, the second stage is constructed with
equality constraints and the L\"owner-John ellipsoid will
degenerate. Consider the equality constraint
\begin{equation}
  \label{eq:444}
  Wy_k=h_k-T_kx, y_k\geq 0, k\in \{1,\ldots,K\}
\end{equation}
To satisfy the slater's condition, we replace the constraints
\eqref{eq:444} by the following
\begin{equation} 
  \label{eq:37}
  Wy_k\leq h_k-T_kx+\epsilon\mathbf{1}, -Wy_k\leq T_kx-h_k+\epsilon\mathbf{1}, y_k\geq 0, k\in \{1,\ldots,K\}
\end{equation} where $\mathbf{1}\in \R^{\ell}$ with all the components
are valued 1 and $\epsilon>0$ is a small enough value. 

We now present the well-known John's theorem.
\begin{theorem}
  \label{sec:ellips-assoc-with-2}
  Let $C$ be a non-empty and bounded polyhedron, which satisfies the
  Slater's condition, in $\R^{n+m}$. There exists a unique
  ellipsoid, the L\"owner-John ellipsoid, which is the smaximum value
  inscribed ellipsoid of C. 
\end{theorem}
We find the proof of Theorem \ref{sec:ellips-assoc-with-2} in many
articles, e.g., \cite{boyd2004convex}. For each $P_k$, we use
$(S(k),c(k))$ to represent the L\"owner-John ellipsoid, where
$c(k)$ is the center of the ellipsoid, and $S(k)$ is the symmetric
positive definite matrix. The calculation of L\"owner-John ellipsoid
is to solve a semi-definite program, which is a convex optimization problem. The
objective function is to minimize the logarithmic function of the
determinant of $S^{-1}(k)$. The computational complexity is
$O[(n+m)^3]$ (see \cite{boyd2004convex}), which indicates the
availability of efficient solution methods. For $P_k$, we have
\begin{equation}
  \label{eq:13}
   (S(k), c(k))\subset P_k\in \R^{n+m}
\end{equation}
and $(S(k), c(k))$ is unique to $P_k$ because Theorem
\ref{sec:ellips-assoc-with-2} proves that the mapping from
each $P_k$ to $(S(k),c(k))$ is a one-to-one correspondence. 

\section{Measuring similarity}
\label{sec:measuring-similarity}

We now construct a measure of similarity among distinct $P_k,
k=1,\ldots,K$. Consider two distinct scenario polyhedrons, $P_j$ and
$P_k$, such that $j,k\in\{1,\ldots,K\}$ and $j\neq k$. $P_j$'s
L\"owner-John ellipsoid is $(S(j),c(j))$ and $(S(k),c(k))$ is $P_k$'s
L\"owner-John ellipsoid. We define
\begin{equation}
\label{eq:27}
 \sigma(k):=\min_{x_k,y_k}\{c'x_k+q'_ky_k: [x_k;y_k]\in
 (S(k),c(k)), x_k, y_k\geq 0\}
\end{equation} and its optimal solution is $[x_k^*;y_k^*]$. By the
feasible region argument, we have 
\begin{equation}
  \label{eq:28}
  -\infty \leq \kappa(k)\leq \sigma(k)\leq \infty
\end{equation}
$\sigma(k), \textrm{ and } \kappa(k), k=1,\ldots,K$ would be
easily calculated. We also assume
\begin{assumption}
  \label{sec:measuring-similarity-1}
  There exists a value $\mathbb{D}<\infty$ that
  $|\sigma(k)-\kappa(k)|\leq \mathbb{D}, k=1,\ldots,K$.
\end{assumption}

The measure of similarity involving the scenario $\xi^k$ would be
$(\kappa(k), \sigma(k))$. Such a poll of coordinates ($k=1,\ldots,K$) formulates
a scatter plot. By using a fine grid setting, we would be able to choose
representative scenarios. According to Theorem \ref{sec:ellips-assoc-with-2}, the
mapping from $P_k$ to $(S(k),c(k))$ would be the one-to-one
correspondence, which means that although comparing $\xi^j$ and
$\xi^k$ are difficult, the
comparison becomes between two L\"owner-John ellipsoids, $(S(k),c(k))$
and $(S(j), c(j))$. Since L\"owner-John ellipsoid feasible region is
strictly convex, both optimal solutions $[x_k^*;y_k^*]$ and
$[x_j^*;y_j^*]$ and optimal values $\sigma(k)$ and $\sigma(j)$ are
unique. Thus, two distinct scenarios $\xi^k$ and $\xi^j$ are
equivalently converted to two distinct optimal values $\sigma(k)$ and
$\sigma(j)$, respectively. We scalarize the non-comparable scenarios
to comparable real values.

For any feasible first-stage decision $x_0$, we have
\begin{equation}
  \label{eq:29}
  Q(x_0,\xi^k):=\min\{q_k'y: [x_0;y]\in P_k\}\tag{recourse value}
\end{equation} and it is true that $\kappa(k)\leq Q(x_0,\xi^k)$ for
any $x_0$. According to \eqref{eq:28}, we need both $\sigma(k)$ and
$\kappa(k)$ to bound the recourse value of any first-stage decision from both
sides. Thus, we use the coordinate $(\kappa(k), \sigma(k))$ as the
similarity measure. 

The similarity measure $(\kappa(k),\sigma(k))$ will help remove 
stage dependency. In other words, when solving the stochastic program
using the SAA, all
sampled $K$ scenarios must share the same first-stage decision. We
have a dilemma, in that we need to quantify the similarity of $K$ iid
scenarios involving the same first-stage decision, while the stochastic
program needs to find the first-stage decision in order to optimize the
objective. We need a similarity measure, which is consistent to any
first-stage decisions. By using the coordinate $(\kappa(k),\sigma(k))$, for any
first-stage decision $x \in X$, we have
\begin{equation}
  \label{eq:31}
  \kappa (k)\leq Q(x,\xi^k)\leq \sigma(k)\nonumber
\end{equation} where $k=1,\ldots,K$. $(\kappa(k),\sigma(k))$ serves as a
spectrum, and regardless of the first-stage decision, the recourse
value will be well positioned within the spectrum. We define that for
each $x\in X$, 
\begin{equation}
  \label{eq:16}
  \epsilon_k(x):=\sigma(k)-Q(x,\xi^k)\geq 0
\end{equation} and for each pair of distinct scenarios $j$ and $k$,
with Assumption \ref{sec:measuring-similarity-1}, we have, for any
$x\in X$, 
\begin{equation}
  \label{eq:17}
  |\epsilon_j(x)-\epsilon_k(x)|\leq \mathbb{D}
\end{equation}

Consider two distinct scenarios, which could be
significantly different in many geometric aspects. However, the
recourse value under $\xi^k$ of any first-stage decision $x$, $Q(x,\xi^k)$, shares unique upper and
lower bounds associated with the scenario $\xi^k$. Regarding the practice of the SAA, when two seemingly
distinct scenarios' recourse values $Q(x,\xi^k)$ and $Q(x,\xi^k)$
share the similar upper and lower bounds, these two scenarios will
yield similar impact to the sample average
$\dfrac{1}{K}\sum_{k=1}^KQ(x, \xi^k)$ and therefore they may be
possibly considered as similar. In other words, we reduce the iid sample
size by clustering similar scenarios, quantified by the spectrum of
the recourse value. We show in the next section that our
clustering method will preserve the convergence in probability and consistency
results of the SAA. 

\section{Clustering similar scenarios}
\label{sec:clust-simil-scen}

\subsection{Preserving convergence rate and consistency}
\label{sec:pres-conv-rate}

Consider the SAA problem with a sample of size $K$, $\xi^1,\ldots,\xi^K$,
\begin{alignat}{2}
  \label{eq:14}
  \min_x&\ c'x+\dfrac{1}{K}\sum_{k=1}^KQ(x,\xi^k)\nonumber\\
  \textrm{subject to}&\ Ax\leq b, x\geq 0, x\in X\subset \R^n
\end{alignat} with the optimal solution $x_K^*$ and the optimal value
$\nu_K^*:=\nu_K(x_K^*)$. Consider two similar scenarios $\xi^j$ and
$\xi^k$, such that $|\kappa(j)-\kappa(k)|\leq \dfrac{\delta}{2}$ and
$|\sigma(j)-\sigma(k)|\leq \dfrac{\delta}{2}$ for a $\delta>0$. We cluster
 scenario $j$ with scenario $k$. After clustering, the iid
sample $\xi^1,\ldots,\xi^K$ becomes 
$\xi^1,\ldots,\xi^{j-1},\xi^k, \xi^{j+1},\ldots,\xi^K$ with the probability
$\dfrac{1}{K}$. That is, scenario $\xi^k$ has a probability of
$\dfrac{2}{K}$ because we replace $\xi^j$ with $\xi^k$. We have: 
\begin{alignat}{2}  
  \label{eq:15}
  \min_x&\ c'x+\dfrac{1}{K}[\sum_{k\neq j}^K\bigg(\sigma(k)-\epsilon_k(x)\bigg)+\sigma(k)-\epsilon_j(k)]\nonumber\\
  \textrm{subject to}&\ Ax\leq b, x\geq 0, x\in X\subset \R^n
\end{alignat}
The difference between the objectives of \eqref{eq:14} and
\eqref{eq:15} at $x$ is bounded by
\begin{equation}
  \label{eq:18}
  \dfrac{1}{K}\bigg|\sigma(j)-\sigma(k)+\epsilon_k(x)-\epsilon_j(x)\bigg|\leq
  \dfrac{\delta}{2K}+\dfrac{\mathbb{D}}{2K}. 
\end{equation}

For a $\delta>0$, we cluster a certain number of scenarios,
$K(\delta)$. There are $K-K(\delta)$ scenarios left for the SAA. Let
$\mathcal{J}$ represent the set of indices of the scenarios being
clustered, and let $\mathcal{K}$ represent the set of indices of the scenarios
remaining in the model. We refer to scenarios in $\mathcal{K}$ as the
representative scenarios. The optimal solution is $\tilde x_K^*$, and
the optimal value becomes $\tilde \nu_K^*:=\tilde \nu(\tilde
x_K^*)$. We use $\tilde \nu_K(x)$ to represent the objective function of
the reduced sample. The value of $|\tilde \nu_K(x)-\nu_K(x)|$ is the measure of
the solution quality in comparison to the SAA. 
\begin{equation}
  \label{eq:19}
  |\tilde \nu_K(x)-\nu_K(x)|\leq \dfrac{\delta K(\delta)}{2K}+\dfrac{K(\delta)\mathbb{D}}{2K}
\end{equation}
The goal of our clustering approach is to reduce the sample size of an
iid sample and to preserve the consistency and convergence in probability of
the SAA.
\begin{theorem}
  \label{sec:clust-simil-scen-1}
  For an iid sample of size $K$, we cluster $K(\delta)$ times, then,
  \begin{equation}
    \label{eq:20}
    |\tilde \nu_K^*-\nu_K^*|\rightarrow 0
  \end{equation} as $K\rightarrow \infty$ with probability 1. 
\end{theorem}
\begin{proof}
  By definition, we have
  \begin{equation}
    \label{eq:21}
    \tilde \nu_K(\tilde x_K^*)\leq \tilde \nu_K(x_K^*) \textrm{ and }
    \nu_K(x_K^*)\leq \nu_K(\tilde x_K^*)
  \end{equation}
  By \eqref{eq:19}, when $K$ is large enough, and for any $\epsilon>0$,
  we have
  \begin{equation}
    \label{eq:23}
    \tilde \nu_K(\tilde x_K^*)\leq \tilde \nu_K(x_K^*)\leq
    \nu_K(x_K^*)+\epsilon \textrm{ and }\nu_K(x_K^*)\leq \nu_K(\tilde
    x_K^*)\leq \tilde \nu_K(\tilde x_K^*)+\epsilon\nonumber
  \end{equation}
  Both $\tilde \nu_K(x)$ and $\nu_K(x)$ are convex function with
  respect to $x$, and when $K$ is large enough, we have $|\tilde
  \nu^*_K-\nu_K^*|\leq \epsilon$ with probability 1. 
\end{proof}
There is another appearance of the above result: for a sample size
of $K$, we cluster up to $K(\delta)$ times, and we have
\begin{equation}
  \label{eq:24}
  \lim_{K\rightarrow \infty}\mathbb{P}(|\tilde
  \nu_K^*-\nu_K^*|>\epsilon)=0\nonumber 
\end{equation}
We complete the analysis regarding the consistency of the SAA.

We now present the impact of our clustering method on the convergence
in probability of the SAA. The clustering method will reduce the
original sample size $K$ to its fraction. Let
$\mathcal{S}_K^\epsilon$ and $\mathcal{S}^\epsilon$ be the sets of the 
$\epsilon$-optimal solutions of the SAA and the original problems,
respectively. Both $\mathcal{S}_K^\epsilon$ and $\mathcal{S}^\epsilon$
are non-empty and finite for any $\epsilon>0$. Let $\mathcal{S}$
represent the set of optimal solutions of the original problem. When
pursuing different 
accuracy in the SAA and the original problem with an accuracy of $\gamma>0$
and an $\epsilon>0$, respectively, such that $\gamma \leq \epsilon$, the
event $\{\mathcal{S}_K^\gamma\subset \mathcal{S}^\epsilon\}$ means
that the solution of the SAA provides an $\epsilon$-optimal solution
for the original problem. We need the following definition and
assumption.
\begin{definition}
  \label{sec:clust-simil-scen-2}
  Let $X:=\{x|Ax\leq b,x\geq 0\}$. $u(x)$ is a mapping from
  $X\backslash  \mathcal{S}^\epsilon$ into the set $\mathcal{S}$, in
  which 
  $u(x)\in \mathcal{S}$ for all $x\in X\backslash
  \mathcal{S}^\epsilon$, such that for $\epsilon^*:=\min_{X\backslash
    \mathcal{S}^\epsilon}\nu(x)-\nu^*$, $\epsilon^*\geq \epsilon$,
  \begin{equation}
    \label{eq:25}
    \nu(u(x))\leq \nu(x)-\epsilon^*\textrm{ for all }x\in X\backslash \mathcal{S}^\epsilon\nonumber
  \end{equation}
\end{definition}
\begin{assumption}
  \label{sec:clust-simil-scen-3}
  For every $x\in X\backslash \mathcal{S}^\epsilon$, the moment-generating function of the random variable
  $Y(x,\xi):=c'u(x)+Q(u(x),\xi)-c'u-Q(x,\xi)$ is finite valued in a
  neighborhood of $t=0$. 
\end{assumption}
We thus have the following result, whose proof is in
\cite{shapiro2009lectures}.
\begin{theorem}
  \label{sec:clust-simil-scen-4}
  Let $\epsilon$ and $\gamma$ be non-negative numbers, such that
  $\gamma\;eq \epsilon$. Then
  \begin{equation}
    \label{eq:26}
    1-\mathbb{P}(\mathcal{S_K^\gamma}\subset \mathcal{S}^\epsilon)\leq
    M e^{-K\eta(\gamma,\epsilon)}
  \end{equation} where $\eta(\gamma,\epsilon):=\min_{x\in
    X\backslash\mathcal{S}^\epsilon}\mathcal{I}_x(-\gamma)$ and
  $\mathcal{I}_x(\cdot)$ denote the rate function of the random
  variable $Y(x,\xi)$. With Assumption 
  \ref{sec:clust-simil-scen-3}, $\eta(\gamma,\epsilon)>0$, and $M$ is
  the number of scenarios of the original problem. 
\end{theorem}

After clustering the scenarios, the reduced sample is no longer iid, 
and we need Theorem \ref{sec:clust-simil-scen-4} to obtain the new
convergence rate results. Let
$\beta:=\dfrac{K(\delta)\delta}{2K}+\dfrac{K(\delta)\mathbb{D}}{2K}$
be the error bound of the clustering $K(\delta)$ scenarios, and we solve the
SAA model for a $\tau$-optimal solution. We denote
$\mathcal{S}_{K,\delta}^{\tau,\beta}$ as the set of resulting optimal
solution, such that $\mathcal{S}_{K,\delta}^{\tau,\beta}\subset
\mathcal{S}_K^\gamma$. We have
\begin{theorem}
  \label{sec:clust-simil-scen-5}
  For an iid sample of size $K$ and $\delta>0$, we cluster $K(\delta)$
  scenarios with the error bound $\beta$, and we solve the SAA with a 
  reduced sample and with an accuracy of $\tau$. If $\mathcal{S}_{K,\delta}^{\tau,\beta}\subset
  \mathcal{S}_K^\gamma$ is true, then
  \begin{equation}
    \label{eq:32}
    \lim_{K\rightarrow \infty}\sup\dfrac{1}{K}\log\bigg[1-\mathbb{P}(\mathcal{S}_{K,\delta}^{\tau,\beta}\subset
\mathcal{S}^\epsilon)\bigg]\leq -\eta({\gamma,\epsilon})
  \end{equation}
\end{theorem}
\begin{proof}
  Since $\mathcal{S}_{K,\delta}^{\tau,\beta}\subset
  \mathcal{S}_K^\gamma$, we have
  \begin{equation}
    \label{eq:33}
    \mathbb{P}(\mathcal{S}_{K,\delta}^{\tau,\beta}\subset
  \mathcal{S}^\epsilon)\geq \mathbb{P}(\mathcal{S}_K^\gamma\subset \mathcal{S}^\epsilon)
\end{equation} and we apply Theorem \ref{sec:clust-simil-scen-5},
\begin{equation}
  \label{eq:34}
  1-\mathbb{P}(\mathcal{S}_{K,\delta}^{\tau,\beta}\subset
  \mathcal{S}^\epsilon)\leq M e^{-K\eta(\gamma,\epsilon)}.
\end{equation}
Applying \eqref{eq:32} and \eqref{eq:34}, we reach the result. 
\end{proof}

We note that the above result is developed for the stochastic program,
which relaxes the integer constraints, if necessary. Because all of the
feasible sets are bounded, non-empty, and non-degenerate, we can use
similar arguments to draw conclusions regarding the convergence
in probability and consistency for stochastic integer programs. 
We now present the clustering approach by steps.
\begin{enumerate}[Step 1.]
\item Generate a large enough iid sample of size $K$. By relaxing
  the integer variables, if necessary, we have $K$ scenario polyhedrons.
\item Calculate the L\"owner-John Ellipsoid for each scenario
  polyhedron.
  \item Calculate $\sigma(k)$, and $\kappa(k), k=1,\ldots,K$
    for each scenario polyhedron. 
  \item Determine the value of $\delta$, which will result in $K(\delta)$
    scenarios removed from the original sample. The choice of $\delta$
    would be problem-dependent. 
  \item Cluster all scenarios, such that $|\sigma(k)-\sigma(j)|\leq
    \dfrac{\delta}{2}$ and $|\kappa(k)-\kappa(j)|\leq
    \dfrac{\delta}{2}$, and replace them with one representative scenario
    $\xi^k, k\in \mathcal{K}$.
    \item Solve the SAA with the reduced sample. 
\end{enumerate}

The above method performs very well regarding  
solution quality for the problems in Section
\ref{sec:numerical-results}. In the next subsection, we present the
connection 
between our method with the other sampling methods. Basically, our
method and other sampling methods share the same goal to solve the
tractability issue. 

\subsection{Connection with other sampling methods}
\label{sec:conn-with-other}

Our method of clustering scenarios are well connected with other
sampling methods surveyed in \cite{homem2014monte} that all sampling
methods, including our clustering method in this paper, would reduce
the sample size for better tractability and high-quality
solution. The SAA method solves a deterministic optimization problem
as the approximate to the original stochastic programming. The SAA
method using iid samples has an advantage over other 
non-iid sampling methods that it simplifies the underlying mathematics
of many statistical method and it assures 
$\dfrac{1}{K}\sum_{k=1}^KQ(x,\xi^k)$ to follow a normal
distribution. The SAA method using iid samples also has nice
theoretical results such as convergence in probability and
consistency. However, the required sample size may be overwhelmingly
large and the scale of the resulting deterministic optimization
problem may easily be out of control. Thus, the non-iid
sampling methods are developed to find good approximations to the original
distribution while not raising the tractability issue.

Most sampling methods work on the distribution side that the
researchers attempt to find a good replacement to the original
distribution. If the goal is achieved, the non-iid sample will be used
to calculate $\dfrac{1}{K}\sum_{k=1}^KQ(x,\xi^k)$. Our method,
however, focuses directly on the value of the recourse function of any
first-stage decision, $Q(x,\xi^k)$. Instead of finding a good
replacement to the original distribution, we will sample the scenarios
which lead to representative values of $Q(x,\xi^k)$, which are bounded
by the coordinates. Thus, the
difference between our method and other non-iid sampling method is
that our method is a recourse-value-oriented method.

Such a difference will disappear in a newsvendor problem as follows.
\begin{example}
  Consider a seller that must choose the amount $x$ of inventory to
  obtain at the beginning of a selling season. The decision is made
  only once, i.e., there is no opportunity to replenish inventory
  during the selling season. The demand $\xi$ during the selling
  season is a nonnegative random variable with cumulative distribution
  function $F$. The cost of obtaining inventory is $c$ per unit. The
  product is sold at a given price $r$ per unit during the selling
  season, and at the end of the season unsold inventory has a salvage
  value of $v$ per unit. The seller wants to choose the amount $x$ of
  inventory that solves
  \begin{equation}
    \label{eq:22}
    \min_x\ \{g(x)=\E[cx-r\min\{x,\xi\}-v\max\{x-\xi,0\}]\}
  \end{equation}
\end{example}
If we present the newsvendor problem in the form of two-stage
stochastic programming, we have
\begin{equation}
  \label{eq:38}
  \min_x\ cx +\E[Q(x,\xi)]\nonumber
\end{equation}
\begin{equation}
  \label{eq:39}
  \textrm{where}\ Q(x,\xi):=\min_{y^+,y^-}\{(r-c)y^-+(c-v)y^+|
  x-y^++y^-=\xi, y^+,y^-\geq 0\}\nonumber
\end{equation}
For a scenario $\xi^k$, our method will have the following model as a
start.
\begin{equation}
  \label{eq:40}
  \kappa(k):=\min_{x,y^+,y^-}\{cx +(r-c)y^-+(c-v)y^+,x-y^++y^-=\xi^k, y^+,y^-\geq 0\}
\end{equation}
For this problem, the L\"owner-John ellipsoid degenerates and
$\sigma(k)=0$ and $\kappa(k)=cx^k$. The optimal solution is
$x^k=\xi^k$ for \eqref{eq:40} and the coordinate becomes
$(0,cx^k)=(0,c\xi^k)$. Thus, our clustering method will consolidate 
similar scenarios $\xi^k$ and $\xi^j, j\neq k$, such that
$|c\xi^k-c\xi^j|\leq \dfrac{\delta}{2}$. We now realize that for the
newsvendor problem, our clustering method behaves just like other
sampling methods such as the importance sampling and the Quasi-Monte
Carlo method to work on the distribution side. This example shows that
our method shares the same goal as other sampling methods, which is to
approximate the value of $\dfrac{1}{K}\sum_{k=1}^KQ(x,\xi^k)$. Our
clustering method works on a large iid sample for the sake of
preserving the nice theoretical results and then reduces the sample
size by focusing on the recourse-value at any first-stage decision
rather than the original distribution. 

We present a cost-benefit discussion for our clustering method with
respect to the
stochastic program, in particular, the stochastic integer program. For
an iid sample of size $K$, the cost involves solving the $K$ Semi-Definite
Programs (SDP) in order to calculate the L\"owner-John Ellipsoids. Thanks to
the advance of convex optimization, the cost is less of a concern
because the SDP can be solved very quickly and efficiently. Also, the
clustering method can highlight the scenarios, which greatly impact the
recourse value, rather than treating every scenario
evenly. Decision-makers may be more interested in identifying several
scenarios that deserve more attention.

The benefit of our approach is that we greatly reduce the number of
scenarios without compromising solution quality. A reduction in the 
scenarios implies a reduction in the number of decision variables and
constraints. In the following section, we demonstrate that the
clustering method reduces an integer program of 202 integer
variables and 201 integer constraints to another instance of 20
integer variables and 19 integer constraints. Consider the fact that
when the removed decision variables are integers, the benefit of
reducing the problem to its fraction can be easily justified. In
real-world implementation, practitioners can start 
clustering scenarios at an early time in order to ``select''
scenarios from a iid sample of massive scenarios. Once a decision is
demanded by customers at a later time, the practitioners can feed a
much smaller model with the selected scenarios to obtain a
high-quality solution. 

\section{Numerical results}
\label{sec:numerical-results}

In this section, we show three representative problems to show that our method would significantly reduce
the sample size without compromising the convergence or the consistency results of
the SAA. We run the numerical experiments on the platform of Windows 10
with Matlab R2018b and SeDuMi solver packages with the Intel i7 processor
and 16GB RAM.

\begin{example}
We first have the following deterministic integer programming:
\begin{equation}
  \label{eq:4}
  \min_{x_1,x_2}\ \{2x_1+3x_2| x_1+x_2\leq 100, 2x_1+6x_2\geq \xi_1,
  3x_1+3x_2\geq \xi_2, x_1,x_2\geq 0,\ x_1, x_2
  \textrm{ are integers.}\}\nonumber
\end{equation}
where $\xi_1, \xi_2$ are demands. When the demands become random, the
amount of the shortage has to be bought at the prices
$q=(q_1,q_2)'=(7,12)'$. We have the stochastic programming as follows:
\begin{alignat}{2}
  \label{eq:6}
  \min_{x_1,x_2}\ &2x_1+3x_2+\E[Q(x_1,x_2,\xi_1,\xi_2)]\nonumber\\
  \textrm{subject to}\ &x_1+x_2\leq 100,x_1,x_2\geq 0,\ x_1, x_2
  \textrm{ are integers.} 
\end{alignat} where
\begin{alignat}{2}
  \label{eq:7}
  Q(x_1,x_2,\xi_1,\xi_2):=\min\ &7 y_1+ 12 y_2\nonumber\\
  \textrm{subject to}\ &y_1= \hat \xi_1-2x_1-6x_2\\
  & y_2= \hat \xi_2-3x_1-3x_2\nonumber\\
  & y_1,y_2\geq 0,\ y_1, y_2 \textrm{ are integers} \nonumber
\end{alignat} where $\hat \xi_1$ and $\hat \xi_2$ are the realized
demands. 

Let $\xi_1$ and $\xi_2$ be uniform discrete random variables with a
range of 
$310, 311, \ldots, 319$ and $292, 293, \ldots, 301$,
respectively. Let $\xi_1$ and $\xi_2$ be independent. The total
number of scenarios is 100, with even probabilities, and we solve the
model with all of the scenarios,
\begin{equation}
  \label{eq:35}
  x^*=[70;30], \textrm{ with the optimal value of } 231.2. 
\end{equation}
The problem with all of the scenarios is an integer program with 202
variables and 201 constraints. The solution \eqref{eq:35} is the
optimal solution to the original. We use the optimal solution to
benchmark our clustering method. 

We calculate all coordinates $(\sigma(k),\kappa(k))$ for each scenario
polyhedron and we generate the following scatter plot.
\begin{figure}[htbp]
  \centering
  \includegraphics[scale=0.5]{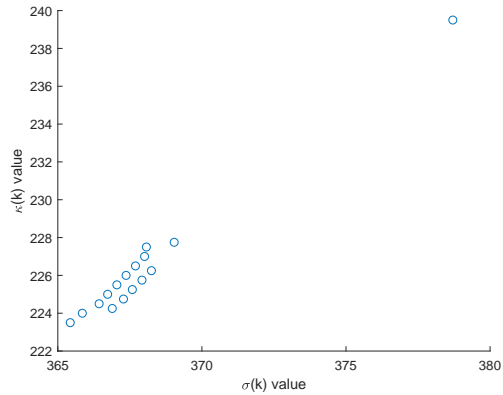}
  \caption{The scatter plot of the coordinates (overlapping points exist)}
  \label{fig:small}
\end{figure}
We adjust $\delta=1.8$ to create grids to cluster 91 scenarios with a ``similar'' spectrum. That
is, we have the following 9 scenarios as the representative
scenarios in Table \ref{tab:ex1}. 
\begin{table}[htbp]
  \centering
  \begin{tabular}[htbp]{c|c|c|c|c|c|c|c|c|c}
    \hline
    $\xi_1$ & 314 & 310 & 310 & 310 & 310 & 315 & 316 & 313 & 310 \\
    \hline
    $\xi_2$ & 301 & 293 & 295 & 296 & 299 & 300 & 300 & 301 & 301\\
    \hline
    Probability & 0.0515 & 0.103 & 0.0412 & 0.3505 & 0.3195 & 0.0309 &
                                                                       0.0618
                                                      & 0.0103 & 0.0309\\ 
    \hline
  \end{tabular}
  \caption{Selected scenarios with $K=100$, $\delta=1.8$, and
    $K(\delta)=91$. }
  \label{tab:ex1}
\end{table}
We solve the problem with only 9 representative scenarios out of 100 scenarios, and the
solution is $[70; 30]$. 
The new problem with the reduced sample has only 20 integer variables
and 19 integer constraints. 
\end{example}

\begin{example}
  This example is called the aircraft allocation problem in
  \cite{dantzig2016linear} and the uncertainty is modeled by 750
  scenarios. We generate all these scenarios and calculate their
  L\"owner-John ellipsoids and coordinates. We plot these values in
  Figure \ref{fig:air}. 
\begin{figure}[htbp]
  \centering
  \includegraphics[scale=0.5]{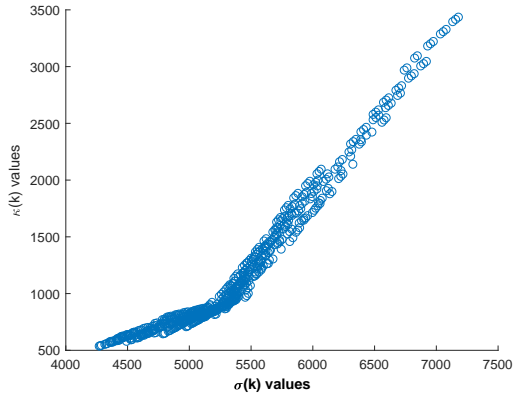}
  \caption{Coordinates of 750 scenarios of the aircraft allocation
    problem} 
  \label{fig:air}
\end{figure}
The ranges of $\sigma(k)$ and $\kappa(k)$ are $[4260,7180]$ and
$[530,3440]$, respectively. We adjust the number bins to evenly cover
both ranges and create grids. Scenarios falling into the same grid
will be treated as similar scenarios. When we set $\delta=29$, we have
$100\times 100$ grids and we consolidate 415 scenarios. Using the remaining
335 scenarios, we obtain the same optimal solution as the published
result. Furthermore, when we set $\delta=290$, we have only $10\times 10$ grids and we
treat the scenarios, whose coordinates fall to the same grid, as
similar scenarios. After clustering similar scenarios, we have 91
scenarios left. Using this reduced sample of 91 scenarios only, we
still obtain the same optimal solution as the published result.   
\end{example}

\begin{example}
We now present the result of another well-known stochastic program,
e.g., ``LandS.'' The problem description is in
\cite{linderoth2006empirical}, and its second-stage decisions are
integers. Our clustering method is applied to a sample of size
$K=20,000$. Without
clustering, the SAA will use an iid sample of size $K=20,000$, which
implies that the mixed-integer program has 240,004 variables, and
240,000 variables among them are integers. It also indicates 140,002 mixed
but mostly integer constraints. When we cluster half of the 
scenarios, the SAA with the reduced sample will have only 120,004
variables and 70,002 constraints. If we cluster scenarios with
$\delta=0.05$, we cluster 18,159 scenarios, such that the equivalent mixed-integer
program will have 22,096 variables and 12,889 constraints, i.e., a
reduction of 89\% scenarios, 90\% variables and constraints. 

The advantage of our method is significant. On an average computer,
the computational time to solve SDP for the coordinate $(\kappa(k),
\sigma(k))$ for each scenario polyhedron is less than 20 seconds for the 
LandS. For
$K=20,000$ scenarios, it will take 240,000 seconds, i.e., less than 67
hours. Since clustering can be deployed to computer cluster, and all of 
the computational tasks could be deployed in parallel, the time would be
greatly shortened. Supposing we have a computer cluster of 10 average
computers, it will take less than 7 hours. The clustering time could
be further shortened by coding the convex optimization solver, which
currently uses Matlab, with more efficient languages. The reduced problem is
significantly friendlier to solvers regarding the scale. Given the
well-known difficulty of the integer program, such a reduction in the 
problem scale is always well justified.
\end{example}

\section{Conclusion}
\label{sec:conclusion}

In this paper, we propose an improvement to the SAA method and the key
idea is to attach a measure of similarity, the coordinate 
$(\sigma(k),\kappa(k))$, to each sampled scenario $\xi^k$. We cluster
the scenarios of similar measures to 
reduce the sample size. We show that the clustering method inherits
both the consistency and convergence in probability of the SAA. The clustering
method will significantly reduce a large enough sample to a small, but
representative one to deliver a timely solution without compromising 
solution quality. The implementation of clustering would be a 
distributed computer cluster, in which the auxiliary computational
tasks, e.g., calculating the L\"owner-John Ellipsoid, and calculating
the coordinate, would be completed by low-cost computers
deployed in parallel rather than expensive supercomputers. The benefit
of clustering is that it reduces the scale of the stochastic program to its
fraction. In numerical examples, nearly 90\% of the integer variables
and constraints are clustered, i.e., removed. Also, the clustering
method will highlight a subset of scenarios requiring more
attention from decision-makers because these scenarios will
generate a significant impact on the optimal solution, compared to the
remaining scenarios.  

This method is worth more intensive testing because not only it
identifies a massive number of scenarios with uniquely valued
coordinates, but also it preserves the nice theoretical results of
SAA. This method would be established as an alternative way of
sampling to solve the tractability issue of stochastic
programming. Furthermore, this method may lead to a distributed
computational infrastructure to solve stochastic programming. The
practitioners can start the solving process long before when the 
solution is demanded. The coordinates of sampled scenarios would be
separately calculated on low-cost computers, which are deployed as
clusters. The solution quality will be continuously improved as more
iid scenarios are processed. 

%
%
%


{\bf ACKNOWLEDGMENT}\ \\
The authors gratefully acknowledge the continued
  support of the School of Business Administration, University of Dayton.





\bibliographystyle{ormsv080}
\bibliography{spa}

\end{document}